\documentclass{article}[12pt]
\usepackage{amsmath}
\usepackage{amsfonts}
\usepackage{amssymb}
\usepackage{amsthm }
\usepackage{footnote}
\usepackage{cleveref}
\usepackage{natbib}
\usepackage{fullpage}

\usepackage{enumerate}
\newtheorem{Thm}{Theorem}[section]

\newtheorem{Lem}[Thm]{Lemma}

\newtheorem{Pro}[Thm]{Proposition}
\makeatletter
\def\blfootnote{\xdef\@thefnmark{}\@footnotetext}
\makeatother

\theoremstyle{definition}
\newtheorem{Def}[Thm]{Definition}

\theoremstyle{remark}
\newtheorem{Rem}[Thm]{\bf{Remark}}

\newcommand{\ConvD}{\overset{d}{\rightarrow}}
\newcommand{\ConvFDD}{\overset{f.d.d.}{\longrightarrow}}
\newcommand{\Cov}{\mathrm{Cov}}
\newcommand{\Var}{\mathrm{Var}}
\newcommand{\E}{\mathrm{E}}

\title{Multivariate limits of multilinear polynomial-form processes with long memory}
\author{Shuyang Bai, Murad S. Taqqu}
\date{\today}
\begin{document}
\maketitle
\begin{abstract}
We consider the multilinear polynomial-form process
 \[X(n)=\sum_{1\le i_1<\ldots<i_k<\infty}a_{i_1}\ldots a_{i_k}\epsilon_{n-i_1}\ldots\epsilon_{n-i_k},\] obtained by applying a multilinear polynomial-form filter to i.i.d.\ sequence $\{\epsilon_i\}$ where $\{a_i\}$ is regularly varying. The resulting sequence $\{X(n)\}$ will then display either short or long memory. Now consider a vector of such $X(n)$, whose components are defined through different $\{a_i\}$'s, that is, through different multilinear polynomial-form filters, but using the same $\{\epsilon_i\}$. What is the limit of the normalized partial sums of the vector? We show that the resulting limit is either a) a multivariate Gaussian process with Brownian motion as marginals, or b) a multivariate Hermite process, or  c) a mixture of the two. We also identify the independent components of the limit vectors.
\end{abstract}

\blfootnote{
\begin{flushleft}
\end{flushleft}
\textbf{2010 AMS Classification:} 60G18, 60F05\\
}
\section{Introduction}
A linear process is generated by applying a linear time-invariant  filter to i.i.d. random variables. 
A common model for stationary \emph{long-range dependent} (LRD) (or \emph{long-memory}) time series is a causal linear process with regularly varying coefficients as the lag tends to infinity, namely,  $X(n)=\sum_{i=1}^\infty a_i \epsilon_{n-i}$, where the $\epsilon_i$'s are i.i.d. with mean 0 and finite variance, and the coefficients satisfy  $a_i=i^{d-1}L(i)$ with  $0<d<1/2$  and $L$ is a slowly varying function at infinity (i.e., $L(x)>0$ when $x$ is large enough and $\lim_{x\rightarrow\infty}L(\lambda x)/L(x)=1$ $\forall \lambda>0$). Note that  $0<d<1/2$ implies $\sum_{i=1}^\infty {|a_i|}=\infty$ but $\sum_{i=1}^\infty a_i^2<\infty$, so $X(n)$ is well-defined in $L^2$ sense.
It is well-known that the autocovariance $\gamma(n)$ of $X(n)$ is regularly varying with power $2d-1$, and that the partial sum of $X(n)$ when suitably normalized converges to fractional Brownian motion with  Hurst index $H=d+1/2$. See for example Chapter 4.4 of \citep{giraitis2009large}.

A family of processes related to multilinear processes are the so-called \emph{multilinear polynomial-form processes} (or \emph{discrete-chaos processes}), which are defined as
\begin{equation}\label{eq:Def Poly Process}
X(n)=\sum_{1\le i_1<\ldots<i_k<\infty}a_{i_1}\ldots a_{i_k}\epsilon_{n-i_1}\ldots\epsilon_{n-i_k},
\end{equation}
where $\sum_{i=1}^\infty a_i^2<\infty$ and $\epsilon_i$'s are i.i.d., and the $k>0$ is the \emph{order}. $X(n)$ is also said to belong to a \emph{discrete chaos} of order $k$. The multilinear polynomial-form process $X(n)$ can be viewed as generated by  nonlinear filters applied to i.i.d. random variables when $k>1$. We call such a nonlinear filter defined in (\ref{eq:Def Poly Process}) a \emph{multilinear polynomial-form filter}.
Such a process often arises from considering a polynomial of a linear process (see, e.g., \citep{surgailis1982zones}).

If $a_i=i^{d-1}L(i)$ with $0<d<1/2$, when $k>1$, that is, except for linear processes, the partial sum of $X(n)$ when suitably normalized no longer converges to a fractional Brownian motion, but depending on $d$ and $k$, it either converges to  a \emph{Hermite process} if $X(n)$ is still LRD, or it converges to a Brownian motion if $X(n)$ is \emph{short-range dependent} (SRD), that is, when the autocovariance of $X(n)$ is absolutely summable.
See \citep{giraitis2009large} for more details.

In Statistics, however, one often needs  convergence when $X(n)$ is a vector rather than a scalar. This leads us to the following question:  if one applies  different multilinear polynomial-form filters to the same i.i.d. sequence $\{\epsilon_i\}$, what is the joint limit behavior of the $J$-vector of the partial sums? More specifically,
assume that $\{\epsilon_i\}$ are i.i.d with  mean 0 and variance 1. Consider the multilinear polynomial-form processes:
\[X_j(n):=\sum_{1\le  i_1<\ldots<i_{k_j}<\infty} a_{i_1,j}\ldots a_{i_{k_j},j}\epsilon_{n-i_1}\ldots\epsilon_{n-i_{k_j}}, ~j=1,\ldots,J,\]
where $k_1,\ldots,k_J$ are orders for $X_1(n),\ldots,X_J(n)$ respectively, $\{a_{i,j}\}$ are regularly varying coefficients.
Let
\begin{equation}\label{eq:Y_j,N}
Y_{j,N}(t)=\frac{1}{A_j(N)}\sum_{n=1}^{[Nt]} X_j(n), ~t\ge 0,
\end{equation}
where $A_j(N)$  is a normalization factor such that $\lim_{N\rightarrow\infty} \Var[Y_{j,N}(1)]=1$, $j=1,\ldots,J$. We want to study the  limit of the following vector process as $N\rightarrow\infty$:
\begin{align}\label{eq:Target}
\mathbf{Y}_N(t):=\left(Y_{1,N}(t),\ldots,Y_{J,N}(t)\right).
\end{align}
Depending on $\{a_{i,j}\}$ and $k_j$, the components of $\mathbf{Y}_N(t)$ can be either purely SRD, or purely LRD, or  a mixture of SRD and LRD.  In \citep{taqqu2012multivariate}, a similar type of problem is considered for  nonlinear functions of a LRD Gaussian process. We show here that the results for multilinear polynomial-form processes are similar to those in \citep{taqqu2012multivariate}. But in the present context, we are able to provide a complete answer to the problem, in contrast  to what happens in \citep{taqqu2012multivariate}, where the mixed SRD and LRD case is stated as a conjecture in some cases.

In addition, we distinguish here between two types of SRD sequences, one involving a linear process ($k=1$) and one involving higher-order multilinear polynomial-form process ($k\ge 2$). For the first type of process, we get dependence with the LRD limit component, while for the second type, we get independence.

The paper is organized as follows. In Section \ref{sec:pre}, some properties of multilinear polynomial-form processes are given and the univariate limit theorems under SRD and LRD are reviewed. In Section \ref{sec:mul}, we state the multivariate convergence results in three cases: a) pure SRD case, b) pure LRD case and c) mixed SRD and LRD case. The result of the general mixed case is stated in \cref{Thm:SRD&LRD}. In Section \ref{sec:proof}, we give the proofs of the results in Section \ref{sec:mul}.

\section{Preliminaries}\label{sec:pre}
In this section, we introduce some facts about multilinear polynomial-form processes as well as the univariate limit theorems for the partial sums.

Suppose that $X(n)$ is the multilinear polynomial-form process in (\ref{eq:Def Poly Process}). Note first, the condition $\sum_{i=1}^\infty a_i^2<\infty$ guarantees that $X(n)$ is well-defined in $L^2$, since
\[
E[X(n)^2]=\sum_{1\le i_1<\ldots<i_k<\infty}a_{i_1}^2\ldots a_{i_k}^2<\infty.
\]
We use throughout a convention $a_i=0$ for $i\le 0$. One can compute the autocovariance of $X(n)$ as:
\begin{equation}\label{eq:CovFormula}
\gamma(n)=\sum_{1\le i_1<\ldots<i_k<\infty}a_{n+i_1}a_{i_1}\ldots a_{n+i_k}a_{i_k},\quad n\in \mathbb{Z}.
\end{equation}
The following proposition describes the asymptotic behavior of $\gamma(n)$  under the assumption:
$
a_i=i^{d-1}L(i),~ i\ge 1,~ 0<d<1/2.
$
\begin{Pro}\label{Pro:AsympACF}
Suppose $\gamma(n)$ is defined in (\ref{eq:CovFormula}), $a_i=i^{d-1}L(i)$, $i\ge 1$ with $0<d<1/2$ where $L$ is slowly varying at infinity.  Then
$\gamma(n)=L^*(n)n^{2d_X-1}$ for some slowly varying function $L^*$ and
\begin{equation}\label{eq:d_X}
d_X=\frac{1}{2}-k(\frac{1}{2}-d).
\end{equation}
\end{Pro}
\begin{proof}
First we claim that as $n\rightarrow\infty$,
\begin{align*}
\sum_{i=1}^\infty a_{n+i}a_{i}\sim  n^{2d-1}B(d,1-2d)L(n)^2,
\end{align*}
where $B(.,.)$ is the beta function.
Indeed, one can check by Potter's bound for slowly varying functions (Theorem 1.5.6 in \citep{bingham1989regular}) and the Dominated Convergence Theorem  that as $n\rightarrow\infty$
\begin{align}\label{eq:sum a}
\frac{1}{L(n)^2n^{2d-1}}\sum_{i=1}^\infty a_{n+i}a_{i}&=\sum_{i=1}^\infty (\frac{i}{n})^{d-1}(1+\frac{i}{n})^{d-1} \frac{L(i)}{L(n)}\frac{L(n+i)}{L(n)}\frac{1}{n}
\\&\rightarrow \int_{0}^\infty u^{d-1}(1+u)^{d-1}du=B(d,1-2d).\notag
\end{align}
Then note that as $n\rightarrow\infty$, $\gamma(n)\sim (k!)^{-1}(\sum_{i=1}^\infty a_{n+i}a_i)^k$ (the diagonal terms with $i_p=i_q$ are negligible as $n\rightarrow\infty$. See also \citep{giraitis2009large} p.109). Now we can deduce that
\[
\gamma(n)=n^{k(2d-1)}L^*(n)=n^{2d_X-1}L^*(n),
\]
where $L^*(n)=(k!)^{-1}B(d,1-2d)^kL(n)^{2k}$.
\end{proof}

\begin{Rem}\label{Rem:SRD LRD boundary}
According to Proposition \ref{Pro:AsympACF},  when $d<\frac{1}{2}(1-\frac{1}{k})$ (or $k(2d-1)<-1$), we have  $\sum|\gamma(n)|<\infty$, and when $d>\frac{1}{2}(1-\frac{1}{k})$, we have $\sum|\gamma(n)|=\infty$. So if we assume $a_i=i^{d-1}L(i),~ 0<d<1/2$, the quantity $\frac{1}{2}(1-\frac{1}{k})$ is the boundary between SRD and LRD.
\end{Rem}

We now define precisely what SRD and LRD mean for a multilinear polynomial-form process $X(n)$, and from then on we use this definition whenever we talk about SRD or LRD.
\begin{Def}\label{Def:SRD LRD}
Let $X(n)$ be a multilinear polynomial-form process given in (\ref{eq:Def Poly Process}) with coefficient $\{a_i\}$, autocovariance $\gamma(n)$ and  order $k$. We say that $X(n)$ is
\begin{enumerate}[(a)]
\item SRD, if for some $d\in\left(-\infty,\frac{1}{2}(1-\frac{1}{k})\right)$ and some constant $c>0$,
\begin{align}\label{eq:SRD Def}
|a_i|\le c i^{d-1},~ i\ge 1,~ \sum_{n=-\infty}^\infty \gamma(n)>0;
\end{align}
\item LRD, if for some $d\in\left(\frac{1}{2}(1-\frac{1}{k}),\frac{1}{2}\right)$ and some $L$  slowly varying at infinity,
\begin{align}\label{eq:LRD Def}
a_i=i^{d-1}L(i),~ i\ge 1,~ \frac{1}{2}(1-\frac{1}{k})<d<1/2.
\end{align}
\end{enumerate}
\end{Def}
\begin{Rem}\label{Rem:a->sum}
The $d$ in (\ref{eq:SRD Def}) and (\ref{eq:LRD Def}) are different. In the SRD case, $\{a_i\}$ is only assumed to decay faster than a power function, which implies $\sum_{n}|\gamma(n)|\le \sum_n (\sum_{i=1}^\infty|a_{n+i}a_i|)^k<\infty$ by (\ref{eq:sum a}), and the particular $d$ chosen will not matter in the limit. While in the LRD case, the regularly varying assumption on $\{a_i\}$ yields  a \emph{memory parameter} $d_X=\frac{1}{2}-k(\frac{1}{2}-d)$  given by (\ref{eq:d_X}), and thus  $d$  plays an important role.
\end{Rem}

Next we consider the cross-covariance between of two  multilinear polynomial-form processes obtained by applying two multilinear polynomial-form filters to the same $\{\epsilon_i\}$. In particular, set
\begin{align}
X_1(n)&=\sum_{1\le i_1<\ldots<i_p<\infty}a_{i_1}\ldots a_{i_p}\epsilon_{n-i_1}\ldots\epsilon_{n-i_p},\label{eq:X_1}\\
X_2(n)&=\sum_{1\le i_1<\ldots<i_q<\infty}b_{i_1}\ldots b_{i_q}\epsilon_{n-i_1}\ldots\epsilon_{n-i_q}.\label{eq:X_2}
\end{align}
$X_1(n)$ and $X_2(n)$ share the same $\{\epsilon_i\}$ but the sequences $\{a_i\}$ and $\{b_i\}$ can be different.
Then the cross-covariance is
\begin{align}\label{eq:CrossCov}
\gamma_{1,2}(n)=\Cov(X_1(n),X_2(0))=\begin{cases}
0& p\neq q;\\
\sum_{1\le i_1<\ldots<i_k<\infty}a_{i_1}b_{n+i_1}\ldots a_{i_k}b_{n+i_k}& p=q=k
\end{cases}
\end{align}
for any $n\in \mathbb{Z}$.

The following result will be used to obtain the asymptotic  cross-covariance structure between the SRD components of $\mathbf{Y}_N(t)$ in (\ref{eq:Target}).
\begin{Pro}\label{Pro:AbsSumCrossCov}
Let $X_1(n)$ and $X_2(n)$ be given as in (\ref{eq:X_1}) and (\ref{eq:X_2}) with $p=q=k$, and are both SRD in the sense of \cref{Def:SRD LRD}. Then the cross-covariance $\gamma_{1,2}(n)=\Cov(X_1(n),X_2(0))$ is absolutely summable:
\begin{equation}\label{eq:AbsSumCrossCov}
\sum_{n=-\infty}^\infty|\gamma_{1,2}(n)|<\infty.
\end{equation}
Moreover, (\ref{eq:AbsSumCrossCov}) implies that as $N\rightarrow\infty$,
\begin{align}\label{eq:LimitCrossCov}
\Cov\left(\frac{1}{\sqrt{N}}\sum_{n=1}^{[Nt_1]}X_1(n), \frac{1}{\sqrt{N}}\sum_{n=1}^{[Nt_2]}X_2(n)\right)\rightarrow (t_1\wedge t_2)\sum_{n=-\infty}^\infty \gamma_{1,2}(n).
\end{align}
In addition, if $k=1$, then
\begin{equation}\label{eq:gamma_1,2 k=1}
\sum_{n=-\infty}^\infty\gamma_{1,2}(n)=\sigma_1\sigma_2,
\end{equation}
where $\sigma_j^2=\sum_{n}\Cov\left(X_j(n),X_j(0)\right)=\lim_{N\rightarrow\infty}\Var\left(\frac{1}{\sqrt{N}}\sum_{n=1}^{[Nt]}X_j(n)\right)$, $j=1,2$.
\end{Pro}
\begin{proof}
Suppose that $\{a_i\}$ and $\{b_i\}$  satisfy the bound in (\ref{eq:SRD Def}) with $d=d_1$  and $d=d_2$ respectively. Using a similar argument as in the proof of \cref{Pro:AsympACF}, one can show that
\begin{align*}
|\gamma_{1,2}(n)|\le  |n|^{k(d_1+d_2-1)} L^*(n)
\end{align*}
for some  function $L^*(n)$ slowly varying at $\pm \infty$. Since by assumption $d_1,d_2<\frac{1}{2}(1-\frac{1}{k})$, which implies that $k(d_1+d_2-1)<-1$, so we have $\sum_n|\gamma_{1,2}(n)|<\infty$.

The proof of (\ref{eq:LimitCrossCov}) follows from the argument of Lemma 4.1 in \citep{taqqu2012multivariate}, after noting that
\[
\Cov\left(\sum_{n=1}^{[Nt_1]}X_1(n), \sum_{n=1}^{[Nt_2]}X_2(n)\right)=\sum_{n_1=1}^{[Nt_1]}\sum_{n_2=1}^{[Nt_2]}\gamma_{1,2}(n_1-n_2).
\]

Now let's prove (\ref{eq:gamma_1,2 k=1}). When $k=1$,  $X_1(n)=\sum_{i=1}^\infty a_i\epsilon_{n-i}$, $X_2(n)=\sum_{i=1}^\infty b_i\epsilon_{n-i}$. Note that by (\ref{eq:SRD Def}) with $k=1$, we have  $\sum_i |a_i|<\infty$ and $\sum_i|b_i|<\infty$. The cross-covariance is $\gamma_{1,2}(n)=\Cov(X_1(n),X_2(0))=\sum_{i=1}^\infty a_{i}b_{i+n}$.
By Fubini, $$\sum_{n=-\infty}^\infty \gamma_{1,2}(n)=\sum_{n=-\infty}^\infty\sum_{i=1}^\infty a_i b_{n+i}=(\sum_{i=1}^\infty a_i)(\sum_{n=1}^\infty b_n).$$
Since $(\sum_{i=1}^\infty a_i)^2=\sum_{n}\gamma_1(n)=\sigma_1^2,$ and  $(\sum_{i=1}^\infty b_i)^2=\sum_{n}\gamma_2(n)=\sigma_2^2$, we get relation (\ref{eq:gamma_1,2 k=1}).
\end{proof}

Let's now review the limit theorems for partial sum of a single multilinear polynomial-form process $X(n)$. Let the notation `` $\ConvFDD$ "  denote convergence in finite-dimensional distributions.
\begin{Thm}\label{Thm:UniSRD}
Suppose that $X(n)$ defined in (\ref{eq:Def Poly Process}) is SRD. Then
\begin{align*}
\frac{1}{A(N)} \sum_{n=1}^{[Nt]}X(n)\ConvFDD B(t),
\end{align*}
where $A(N)$ is a normalization factor to guarantee unit asymptotic variance at $t=1$, and $B(t)$ is the standard Brownian motion. In fact,
$A(N)\sim \sigma\sqrt{N}$ as $N\rightarrow \infty$ with $\sigma^2=\sum_{n}\gamma(n)$.
\end{Thm}

\begin{Thm}\label{Thm:UniLRD}
Suppose that $X(n)$ defined in (\ref{eq:Def Poly Process})  is LRD. Then
\begin{align*}
\frac{1}{A(N)} \sum_{n=1}^{[Nt]}X(n)\ConvFDD Z_d^{(k)}(t),
\end{align*}
where $A(N)$ is a normalization factor to guarantee unit asymptotic variance at $t=1$, and $Z_d^{(k)}(t)$ is the so-called Hermite process defined with the aid of the k-tuple Wiener-It\^o stochastic integral denoted by $I_k(.)$ (\citep{major1981multiple}):
\begin{equation}\label{eq:HermProc}
Z_d^{(k)}(t)=I_{k}(f_{k,d}^{(t)}):=\int'_{\mathbb{R}^k} f_{k,d}^{(t)}(x_1,\ldots,x_k) W(dx_1)\ldots W(dx_k)
\end{equation}
where the prime $'$ indicates the exclusion of the diagonals $x_i= x_j$ for $i\neq j$, $W(.)$ is Brownian random measure,  and
\begin{equation}\label{eq:HermKer}
f_{k,d}^{(t)}(x_1,\ldots,x_k)=a_{k,d}\int_{0}^t\prod_{j=1}^k (s-x_j)_+^{d-1}ds,
\end{equation}
with
\[a_{k,d}=\left(\frac{\left(k(d-1/2)+1\right)\left(2k(d-1/2)+1\right)\Gamma(1-d)^k}{k!\Gamma(d)^k\Gamma(1-2d)^k}\right)^{1/2}.\]
(See \citep{pipiras2010regularization}.)
In fact,
$A(N)\sim c N^{1+(d-1/2)k} L(N)^{k/2}$ as $N\rightarrow \infty$  for some $c>0$.
\end{Thm}
For the proofs of \cref{Thm:UniSRD} and \cref{Thm:UniLRD}, we refer the reader to Chapter 4.8 in \citep{giraitis2009large},  respectively Theorem 4.8.1 and Theorem 4.8.2 \footnote{The results of Chapter 4.8 in \citep{giraitis2009large} do not include a slowly varying function, nor convergence of finite-dimensional distributions in the case of \cref{Thm:UniSRD}. But they can be easily extended.}. One may also compare  \cref{Thm:UniSRD} and \cref{Thm:UniLRD} to their counterparts in the context of nonlinear functions of a LRD Gaussian process, stated as Theorem 2.1 and Theorem 2.2 in \citep{taqqu2012multivariate}.

\section{Multivariate convergence results}\label{sec:mul}
In this section, we state the multivariate joint convergence results for the vector process $ \mathbf{Y}_N(t)$ in (\ref{eq:Target}). Recall that $\mathbf{Y}_N$ is normalized so that the asymptotic variance of every component at $t=1$ equals $1$.

\begin{Thm}\label{Thm:PureSRD}\textbf{Pure SRD Case.}
If all the components in $\mathbf{Y}_N$ defined in (\ref{eq:Target}) are SRD in the sense of (\ref{eq:SRD Def}),  then
\begin{align*}
\mathbf{Y}_N(t)\ConvFDD  \mathbf{B}(t)=(B_1(t),\ldots,B_J(t)),
\end{align*}
where $\mathbf{B}(t)$ is a multivariate Gaussian process with $B_1(t),\ldots,B_J(t)$ being standard Brownian motions with
\begin{align}\label{eq:SRDCov}
\mathrm{Cov}\left(B_{p}(s),B_{q}(t)\right)=(s\wedge t)\frac{\sigma_{p,q}}{\sigma_p\sigma_q},
\end{align}
\begin{align*}
&\sigma_p^2= \sum_{n=-\infty}^\infty \gamma_p(n):=\sum_{n=-\infty}^\infty \Cov(X_p(n),X_p(0)),\\
&\sigma_{p,q}= \sum_{n=-\infty}^\infty \gamma_{p,q}(n):=\sum_{n=-\infty}^\infty \Cov(X_p(n),X_q(0)).
\end{align*}
The normalization $A_j(N)$ in (\ref{eq:Y_j,N}) satisfies $A_j(N)\sim \sigma_j \sqrt{N}$  as $N\rightarrow\infty$.
\end{Thm}
\begin{Rem}
$\sigma_{p,q}$ is well-defined by \cref{Pro:AbsSumCrossCov}.
\end{Rem}
\begin{Rem}
In view of (\ref{eq:CrossCov}) and (\ref{eq:SRDCov}), if all the components of the $\mathbf{Y}_N(t)$ have different order, then the limit components $B_j(t)$ uncorrelated and hence  independent. Otherwise, they are in general dependent and their covariance is given by (\ref{eq:SRDCov}).
\end{Rem}

\begin{Thm}\label{Thm:PureLRD}\textbf{Pure LRD Case.}
If all the components in $\mathbf{Y}_N$ defined in (\ref{eq:Target}) are LRD in the sense of (\ref{eq:LRD Def}) with $d=d_1,\ldots d_J$ respectively, then
\begin{align*}
\mathbf{Y}_N(t)\ConvFDD  \mathbf{Z}^{\mathbf{k}}_\mathbf{d}(t)=(Z_{d_1}^{(k_1)}(t),\ldots,Z_{d_J}^{(k_J)}(t)),
\end{align*}
where $Z_{d_j}^{(k_j)}(t)$ are  Hermite processes sharing the same random measure $W(.)$ in their Wiener-It\^o integral representations.  The normalization $A_j(N)$ in (\ref{eq:Y_j,N}) satisfies
$A_j(N)\sim c_j N^{1+(d_j-1/2)k_j} L(N)^{k_j/2}$ as $N\rightarrow \infty$  for some $c_j>0$. The processes $Z_{d_j}^{(k_j)}$, $j=1,\ldots,J$ are dependent.
\end{Thm}

We  now consider the mixed SRD and LRD case.
\begin{Thm}\label{Thm:SRD&LRD}\textbf{Mixed SRD and LRD  Case.}
Break $\mathbf{Y}_N$ in (\ref{eq:Target}) into 3 parts:
$$\mathbf{Y}_N=(\mathbf{Y}_{N,S_1},\mathbf{Y}_{N,S_2},\mathbf{Y}_{N,L}),$$
where within  $\mathbf{Y}_{N,S_1}$ ($J_{S_1}-$dimensional) every component is SRD and has order $k_{j,S_1}=1$,  within  $\mathbf{Y}_{N,S_2}$ ($J_{S_2}-$dimensional) every component is SRD and has order $k_{j,S_2}\ge 2$, and within $\mathbf{Y}_{N,L}$ ($J_{L}-$dimensional) every component is LRD. Then
\begin{align}
\mathbf{Y}_N(t)=\left(\mathbf{Y}_{N,S_1}(t),\mathbf{Y}_{N,S_2}(t),\mathbf{Y}_{N,L}(t)\right)\ConvFDD  (\mathbf{W}(t), \mathbf{B}(t),\mathbf{Z}_{\mathbf{d}_L}^{\mathbf{k}_L}(t)),\label{eq:General Joint Conv}
\end{align}
where $\mathbf{B}(t):=\left(B_1(t),\ldots,B_{J_{S_2}}(t)\right)$ is the multivariate Gaussian process appearing in \cref{Thm:PureSRD},  $\mathbf{Z}_{\mathbf{d}_L}^{\mathbf{k}_L}(t)$ is the multivariate Hermite process appearing in \cref{Thm:PureLRD}, \begin{equation}\label{eq:B_1}
\mathbf{W}(t)=(W(t),\ldots,W(t)),
\end{equation}
where $W(t)$ is the Brownian motion integrator for defining $\mathbf{Z}_{\mathbf{d}_L}^{\mathbf{k}_L}(t)$  (see (\ref{eq:HermProc})), and $\mathbf{B}(t)$ is independent of $(\mathbf{W}(t),\mathbf{Z}_{\mathbf{d}_L}^{\mathbf{k}_L}(t))$.
\end{Thm}

\begin{Rem}
To understand heuristically why $\mathbf{B}(t)$ and  $(\mathbf{W}(t),\mathbf{Z}_{\mathbf{d}_L}^{\mathbf{k}_L}(t))$ are independent, note that $\mathbf{Y}_{N,S_2}(t)$ belongs to  chaos of order $\ge 2$, and is thus uncorrelated with $\mathbf{Y}_{N,S_1}(t)$ which belongs to first-order chaos, and also uncorrelated with the random noise $\{\epsilon_i\}$ which also belongs to the first-order chaos, and which after summing becomes asymptotically the Brownian measure $W(.)$ defining $\mathbf{Z}_{\mathbf{d}_L}^{\mathbf{k}_L}(t)$.
\end{Rem}

\begin{Rem}
The independence between $\mathbf{B}(t)$ and $\mathbf{Z}_{\mathbf{d}_L}^{\mathbf{k}_L}(t)$ for $k_{j,L}\ge 3$ (the order in LRD component) in the framework of \citep{taqqu2012multivariate}, is only a conjecture.
\end{Rem}

The convergence results in the above theorems are stated in terms of convergence in finite-dimensional distributions, but one can show that in some cases they extend to weak convergence in $D[0,1]^J$ (J-dimensional product space where $D[0,1]$ is the space of C\`adl\`ag functions on $[0,1]$ with uniform metric).

\begin{Thm}\label{Thm:WeakConv}
\textbf{Weak convergence in $D[0,1]^J$.}
\begin{enumerate}
\item \cref{Thm:PureLRD} holds with `` $\ConvFDD$ '' replaced by  weak convergence in $D[0,1]^J$;

\item If the SRD component in \cref{Thm:PureSRD} (or \cref{Thm:SRD&LRD}) satisfies either of the following conditions:
\begin{enumerate}[a.]
\item\label{case:m-dependent} There exists $m\ge 0$, such that the coefficients $a_{i}$ in (\ref{eq:Def Poly Process}) are zero for all $i>m$;

\item\label{case:gaussian} $\{\epsilon_i\}$  are i.i.d.\ Gaussian.

\item\label{case:linear} The order $k=1$ and  $E(|\epsilon_i|^{2+\delta})<\infty$ for some $\delta>0$;

\item\label{case:multilinear} The order $k\ge 2$, $\sum_{i=1}^\infty |a_i|<\infty$ and $E(|\epsilon_i|^{5})<\infty$;

\end{enumerate}
then \cref{Thm:PureSRD} (or \cref{Thm:SRD&LRD}) holds with `` $\ConvFDD$ '' replaced by  weak convergence in $D[0,1]^J$.
\end{enumerate}
\end{Thm}
Note that tightness   in the SRD case results from an interplay between the dependence structure and the finiteness of the moments.
\section{Proofs for the multivariate convergence results}\label{sec:proof}

\subsection{Pure SRD case}
\begin{proof}[Proof of \cref{Thm:PureSRD}]
Following the idea of \citep{giraitis2009large} p.108., we define the truncated multilinear polynomial-form processes:
\begin{align}\label{eq:TruncPolyForm}
X_j^{(m)}(n)=\sum_{1\le i_1<\ldots<i_{k_j}\le m}a_{i_1,j}\ldots a_{i_{k_j},j}~\epsilon_{n-i_1}\ldots\epsilon_{n-i_{k_j}},\quad j=1,\ldots,J,
\end{align}
where  $m>max_j\{k_j\}$.  Note that $X_j^{(m)}(n)$ is  $m$-dependent.
Set $(\sigma^{(m)}_{j})^2=\sum_n\mathrm{Cov}\left(X_j^{(m)}(n),X_j^{(m)}(0)\right)$ (assume  $m$ is large enough so that $\sigma^{(m)}_{j}>0$), and $\sigma_{p,q}^{(m)}=\sum_n \Cov\left(X_p^{(m)}(n),X_q^{(m)}(0)\right)$ which is well-defined due to \cref{Pro:AbsSumCrossCov}.

Set
\begin{align*}
Y_{N,j}(t):=\frac{1}{\sigma_{j}\sqrt{N}}\sum_{n=1}^{[Nt]}X_j(n),\quad Y_{N,j}^{(m)}(t):=\frac{1}{\sigma_j^{(m)}\sqrt{N}}\sum_{n=1}^{[Nt]}X_j^{(m)}(n).
\end{align*}
\cref{Thm:PureSRD} follows if one shows that as $N\rightarrow\infty$,
\begin{align}\label{eq:Truncated Conv}
\mathbf{Y}^{(m)}_N
(t)=:\left(Y^{(m)}_{N,1}(t),\ldots,Y^{(m)}_{N,J}(t)\right)
\ConvFDD \mathbf{B}^{(m)}(t):=\left(B^{(m)}_1(t),\ldots,B^{(m)}_J(t)\right)
\end{align}
where $B^{(m)}_j(t)$'s are Brownian motions with cross-covariance structure:
\begin{align}\label{eq:TruncCrossCov}
\Cov(B^{(m)}_p(t_1),B^{(m)}_q(t_2))=(t_1\wedge t_2)\frac{\sigma^{(m)}_{p,q}}{\sigma_{p}^{(m)}\sigma_{q}^{(m)}},
\quad p,q=1,\ldots,J,
\end{align}
and as $m\rightarrow\infty$,
\begin{align}\label{eq:Sigma Conv}
\sigma_{j}^{(m)}\rightarrow\sigma_{j},\quad   \sigma_{p,q}^{(m)}\rightarrow\sigma_{p,q}
\end{align}
as well as for any $j=1,\ldots, J$ and $t\ge 0$, as $m\rightarrow\infty $,
\begin{align}\label{eq:Trunc Ori}
\mathrm{Var}\left[Y^{(m)}_{N,j}(t)-Y_{N,j}(t)\right] \rightarrow 0
\end{align}
uniformly in $N$.
Indeed, combining (\ref{eq:Truncated Conv}), (\ref{eq:Sigma Conv}) and (\ref{eq:Trunc Ori}), one obtains the desired convergence:
\begin{align*}
\mathbf{Y}_N
(t)=\left(Y_{N,1}(t),\ldots,Y_{N,J}(t)\right)
\ConvFDD \mathbf{B}(t):=\left(B_1(t),\ldots,B_J(t)\right)
\end{align*}

Relations (\ref{eq:Sigma Conv}) and (\ref{eq:Trunc Ori}) can be shown using the same type of arguments in  \citep{giraitis2009large} p.108.
We thus only need to show (\ref{eq:Truncated Conv}) and (\ref{eq:TruncCrossCov}).
By the Cr\'amer-Wold device, it suffices to show that for any $(c_1,\ldots,c_J)\in \mathbb{R}^J$,
\begin{align}\label{eq:CramerWold}
\sum_{j}c_jY_{N,j}^{(m)}(t)=
\frac{1}{\sigma_j^{(m)}\sqrt{N}}\sum_{n=1}^{[Nt]} (\sum_{j}c_jX_j^{(m)}(n))
\ConvFDD \sum_j c_jB^{(m)}_{j}(t)=:G(t)
\end{align}
where  $G(t)$ is a non-standardized Brownian motion. This follows from the fact that the sequence $\{\sum_{j}c_jX_j^{(m)}(n)\}_n$ is m-dependent and  is thus subject to functional central limit theorem (\citep{billingsley1956invariance} Theorem 5.2), which includes convergence in finite-dimensional distributions. The asymptotic cross-covariance  structure (\ref{eq:TruncCrossCov}) follows from \cref{Pro:AbsSumCrossCov}.
\end{proof}

\subsection{Pure LRD case}
\begin{proof}[Proof of \cref{Thm:PureLRD}]
The joint convergence is proved by combining  Theorem 4.8.2.  and Proposition 14.3.3   of \citep{giraitis2009large},  and the arguments leading to them.

The dependence between the limit Hermite processes with different orders is shown  in Proposition 3.1 in \citep{taqqu2012multivariate}.
\end{proof}

\subsection{Mixed SRD and LRD case}

We prove \cref{Thm:SRD&LRD} through a number of lemmas, one lemma implying the next.

\begin{Lem}\label{Lem:SRD and W}
Follow the notations and assumptions in \cref{Thm:SRD&LRD}. Let $X_{j,S_i}^{(m)}(n)$ be the m-truncated multilinear polynomial-form process (see (\ref{eq:TruncPolyForm})) corresponding to the components of $\mathbf{Y}_{N,S_i}$ ($i=1,2$) in \cref{Thm:SRD&LRD}, where the orders satisfy $k_{j,S_1}=1$ and $k_{j,S_2}\ge 2$. Let
\[
Y_{N,j,i}^{(m)}(t):=\frac{1}{\sigma_{j,S_i}^{(m)}\sqrt{N}}\sum_{n=1}^{[Nt]}X_{j,S_i}^{(m)}(n),\quad j=1,\ldots,J_i,~i=1,2,
\]
where (assuming that $m$ is large enough) $0<(\sigma^{(m)}_{j,S_i})^2:=\sum_{n}\Cov(X_{j,S_i}^{(m)}(n),X^{(m)}_{j,S_i}(0))<\infty,~i=1,2.$
Let $W_N(t):=N^{-1/2}\sum_{n=1}^{[Nt]}\epsilon_n$, and $\mathbf{Y}_{N,S_i}^{(m)}(t)=(Y_{N,1,i}^{(m)}(t),\ldots,Y_{N,J_{S_i},i}^{(m)}(t))$, $i=1,2$.
Then
\begin{align}\label{eq:S_1 S_2 W}
\Big(\mathbf{Y}_{N,S_1}^{(m)}(t), \mathbf{Y}_{N,S_2}^{(m)}(t),W_N(t)\Big) \ConvFDD \Big(\mathbf{W}(t), \mathbf{B}^{(m)}(t),W(t)\Big),
\end{align}
where $W(t)$ is a standard Brownian motion,  $\mathbf{W}(t)=(W(t),\ldots,W(t))$ ($J_{S_2}$-dimensional),
$\mathbf{B}^{(m)}(t)$ is as given in (\ref{eq:Truncated Conv}), namely, its components are standard Brownian motions with cross-covariance (\ref{eq:TruncCrossCov}), and $\mathbf{B}^{(m)}(t)$ is independent of $(\mathbf{W}(t),W(t))$.
\end{Lem}
\begin{proof}
Fix any $\mathbf{w}=(a_1,\ldots,a_{J_{S_1}},b_1,\ldots,b_{J_{S_2}},c)\in \mathbb{R}^{J_{S_1}+J_{S_2}+1}$. By Cram\'er-Wold, we want to show that
\begin{align*}
R_N(t;\mathbf{w}):=&\sum_{j}a_jY_{N,j,1}^{(m)}(t)+\sum_{j}b_jY_{N,j,2}^{(m)}(t)+ cW_N(t)\\\ConvFDD
&\sum_{j}a_jW(t)+\sum_{j}b_jB^{(m)}_j(t)+ c W(t)=:G(t),
\end{align*}
where $G(t)$ is a non-standardized Brownian motion whose marginal variance  is the limit of the marginal variance of $R_N(t;\mathbf{w})$.
Note that one can write
\begin{align*}
R_N(t;\mathbf{w})=\frac{1}{\sqrt{N}}\sum_{n=1}^{[Nt]}U_{\mathbf{w}}^{(m)}(t),
\end{align*}
where
\begin{align*}
U_{\mathbf{w}}^{(m)}(n)=\sum_{j=1}^{J_{S_1}}\frac{a_j}{\sigma_{j,S_1}^{(m)}}X_{j,S_1}^{(m)}(n)+\sum_{j=1}^{J_{S_2}}\frac{b_j}{\sigma_{j,S_2}^{(m)}}X_{j,S_2}^{(m)}(n)+c e_n^{(m)}
\end{align*}
with \[e_n^{(m)}=\sum_{i=(m-1)n+1}^{mn}\epsilon_i.\]
Since $\{U_{\mathbf{w}}^{(m)}(n)\}_n$ is m-dependent, the classical functional central limit theorem applies (\citep{billingsley1956invariance}), yielding  in the limit a Brownian motion $G(t)$ for $R_N(t;\mathbf{w})$. Now that the joint normality is shown, we only need to identify the asymptotic covariance structure as $N\rightarrow\infty$ of the left-hand side of (\ref{eq:S_1 S_2 W}) to the covariance structure of the right-hand side of (\ref{eq:S_1 S_2 W}).

The independence between $\mathbf{B}^{(m)}(t)$  and $(\mathbf{W}(t),W(t))$  follows from the uncorrelatedness between $\mathbf{Y}_{N,S_2}^{(m)}(t)$ (involving chaos of order $\ge 2$) and $(\mathbf{Y}_{N,S_1}^{(m)}(t),W_N(t))$ (involving chaos of order 1 only). The asymptotic covariance structure within $\mathbf{Y}_{N,S_2}^{(m)}(t)$ is  given in (\ref{eq:TruncCrossCov}) (apply \cref{Thm:PureSRD} to $\mathbf{Y}_{N,S_2}^{(m)}$).
Hence we are left to show that the asymptotic covariance structure of $(\mathbf{Y}_{N,S_1}^{(m)}(t),W_N(t))$ is that of $(\mathbf{W}(t),W(t))$.
Note that in $(\mathbf{Y}_{N,S_1}^{(m)}(t),W_N(t))$, both $\{X^{(m)}_{j,S_1}(n)\}$ and $\{\epsilon_n\}$ are SRD linear processes. So applying (\ref{eq:LimitCrossCov}) and (\ref{eq:gamma_1,2 k=1}) in \cref{Pro:AbsSumCrossCov} with $\sigma_1=\sigma_2=1$, the desired asymptotic covariance structure is obtained.
\end{proof}

\begin{Rem}\label{Rem:Brownian measure}
\cref{Lem:SRD and W} can be rephrased as follows: we define an \emph{empirical random measure} on a finite interval $\Delta$ as: $W_N(\Delta):=\frac{1}{\sqrt{N}}\sum_{n/N\in \Delta} \epsilon_n$. Then the joint convergence in \cref{Lem:SRD and W} still holds with $W(t)$ replaced by $(W_N(\Delta_1),\ldots,W_N(\Delta_I))$ where $\Delta_i, i=1,\ldots,I$ are disjoint intervals, and $W(t)$ in the limit replaced by $(W(\Delta_1),\ldots,W(\Delta_I))$ where $W(.)$ is the Brownian random measure. Observe  that while (\ref{eq:S_1 S_2 W}) involves convergence in distribution, the limit components $\mathbf{W}(t)$ and $W(t)$ both involve the same Brownian motion $W(t)$.
\end{Rem}

Now we adopt some notations from \citep{giraitis2009large} Chapter 14.3. Let
$S_M(\mathbb{R}^{k})$ be the class of simple functions defined on $\mathbb{R}^k$ supported on a finite number of $1/M$-cubes and vanishing on the diagonals. Suppose that $h$ is a function defined on $\mathbb{Z}^k$ which vanishes on diagonals. Let the polynomial form (or discrete multiple integral) with respect to $h$ be
\begin{equation}\label{eq:Q_k(h)}
Q_k(h)=\sum_{i_1,\ldots,i_k\in \mathbb{Z}}h(i_1,\ldots,i_k)\epsilon_{i_1}\ldots\epsilon_{i_k},
\end{equation}
where $\sum_{i_1,\ldots,i_k}h(i_1,\ldots,i_k)^2<\infty$.
The following lemma plays a key role in the proof of \cref{Thm:SRD&LRD}.
\begin{Lem}\label{Lem:2}
Replace $(\mathbf{Y}^{(m)}_{N,S_1}(t),\mathbf{Y}^{(m)}_{N,S_2}(t),\mathbf{W}_N(t))$ in \cref{Lem:SRD and W} by $(\mathbf{Y}^{(m)}_{N,S_1}(t),\mathbf{Y}^{(m)}_{N,S_2}(t),\mathbf{Q}_N)$, where $\mathbf{Q}_N=\left(Q_{k_1}(h_{1,N}),\ldots,Q_{k_{J_L}}(h_{J_L,N}))\right)$ and
each $Q_{k_p}(h_{p,N})$, $p=1,\ldots,J_L$, is a polynomial-form defined in (\ref{eq:Q_k(h)}) with the same $\{\epsilon_i\}$ as those defining $\mathbf{Y}_{N,S_1}^{(m)}(t)$ and $\mathbf{Y}_{N,S_2}^{(m)}(t)$. Assume that the ``normalized continuous extension'' of  $h_{p,N}$, that is,
\begin{equation}\label{eq:h LRD}
\tilde{h}_{p,N}(x_1,\ldots,x_{k_p}):=N^{k_p/2}h_{p,N}([Nx_1],\ldots,[Nx_{k_p}])
\end{equation}
satisfy that there exists $f_p\in L^2(\mathbb{R}^{k_p})$ for each $p=1,\ldots,J_L$,
\begin{equation}\label{eq:kernelConv}
\lim_{N\rightarrow\infty}\|\tilde{h}_{p,N}-f_p\|_{L^2(\mathbb{R}^{k_p})}\rightarrow 0.
\end{equation}
Now define the limit vector $\big(\mathbf{W}(t), \mathbf{B}^{(m)}(t),\mathbf{I}\big)$ as follows: $\mathbf{W}(t)$ and  $\mathbf{B}^{(m)}(t)$ are as in (\ref{eq:S_1 S_2 W}), independent, and $\mathbf{I}=\left(I_{k_p}(f_p)\right)_{p=1,\ldots,J_L}$, where each Wiener-It\^o integral $I_{k_p}(.)$   has Brownian motion integrator $W(.)$ the same as the Brownian motion $W(t)$ defining $\mathbf{W}(t)$. Then as $N\rightarrow\infty$,
\begin{align}\label{eq:S_1 S_2 I}
\Big(\mathbf{Y}_{N,S_1}^{(m)}(t), \mathbf{Y}_{N,S_2}^{(m)}(t),\mathbf{Q}_N\Big) \ConvFDD \Big(\mathbf{W}(t), \mathbf{B}^{(m)}(t),\mathbf{I}\Big).
\end{align}
\end{Lem}
\begin{Rem}
Observe that
$\mathbf{B}^{(m)}$ is independent of $(\mathbf{W},\mathbf{I})$.
\end{Rem}
\begin{proof}
The lemma is proved by combining \cref{Lem:SRD and W} with the proof of Proposition 14.3.2 of \citep{giraitis2009large}.
By Cram\'er-Wold, we need to show that for any $\mathbf{a}\in \mathbb{R}^{J_{S_1}}$, $\mathbf{b}\in \mathbb{R}^{J_{S_2}}$ and $\mathbf{c}\in \mathbb{R}^{J_{L}}$, as $N\rightarrow\infty$,
\begin{align}\label{eq:target}
<\mathbf{a},\mathbf{Y}^{(m)}_{N,S_1}(t)>+<\mathbf{b},\mathbf{Y}^{(m)}_{N,S_2}(t)>+<\mathbf{c},\mathbf{Q}_N>\ConvFDD
<\mathbf{a},\mathbf{W}(t)>+<\mathbf{b},\mathbf{B}^{(m)}(t)>+<\mathbf{c},\mathbf{I}>,
\end{align}
where $<.,.>$ denotes the Euclidean inner product.

Next following the approximation argument that leads to (14.3.14), (14.3.15) and (14.3.16) in \citep{giraitis2009large}, one can show that for any $\epsilon>0$, there exists $M>0$ and simple functions $f_{p,\epsilon}\in S_M(\mathbb{R}^{k_p})$, $p=1,\ldots,J_L$, such that for all $N\ge N_0(\epsilon)$ where $N_0(\epsilon)$ is large enough,
\begin{align}
&\|Q_{k_p}(h_{p,N})-Q_{k_p}(h_{p,\epsilon,N}))\|_{L^2(\Omega)}\le \epsilon,\label{eq:approx 1}\\
&Q_{k_p}(h_{p,\epsilon,N})\ConvD I_{k_p}(f_{p,\epsilon})\quad\text{as } N\rightarrow \infty,\label{eq:approx 2}\\
&\|I_{k_p}(f_{p,\epsilon})-I_{k_p}(f_p)\|_{L^2(\Omega)}\le \epsilon\label{eq:approx 3},
\end{align}
where $\|.\|_{L^2(\Omega)}$ denotes the $L^2(\Omega)$ norm,
\[
h_{p,\epsilon,N}(j_1,\ldots,j_{k_p}):=N^{-k_p/2}f_{p,\epsilon}(\frac{j_1}{N},\ldots,\frac{j_{k_p}}{N}).
\]
Set
\[\mathbf{Q}_{\epsilon,N}:=\Big(Q_{k_p}(h_{p,\epsilon,N})\Big)_{p=1,\ldots,J_L}\]
and
\[
\mathbf{I}_\epsilon:=\left(I_{k_p}(f_{p,\epsilon})\right)_{p=1,\ldots,J_L}.
\]
Now note that  $Q_{k_p}(h_{p,\epsilon,N})$ is a multivariate polynomial (thus is a continuous function) of random variables of the form $W_N(\Delta_i)$ where $\Delta_i$'s are disjoint finite intervals and $W_N(.)$ is the empirical random measure as given in \cref{Rem:Brownian measure}.  So by \cref{Lem:SRD and W} (with \cref{Rem:Brownian measure})  and the Continuous Mapping Theorem, we have that as $N\rightarrow\infty $,
\begin{align}\label{eq:approxCW2}
<\mathbf{a},\mathbf{S}^{(m)}_{N,1}(t)>+<\mathbf{b},\mathbf{S}^{(m)}_{N,2}(t)>+<\mathbf{c},\mathbf{Q}_{\epsilon,N}>\ConvFDD
<\mathbf{a},\mathbf{W}(t)>+<\mathbf{b},\mathbf{B}^{(m)}(t)>+<\mathbf{c},\mathbf{I}_\epsilon>.
\end{align}
By (\ref{eq:approx 1}) and the Cauchy-Schwartz inequality, we infer that
\begin{align}\label{eq:approxCW1}
\|\left(<\mathbf{c},\mathbf{Q}_{N}-\mathbf{Q}_{\epsilon,N}>\right)\|_{L^2(\Omega)}\le  \|\mathbf{c}\| \|\mathbf{Q}_{N}-\mathbf{Q}_{\epsilon,N}\|_{L^2(\Omega)}\le  \|\mathbf{c}\|\sqrt{J_L}\epsilon,
\end{align}
where $\|.\|$ denotes the Euclidean norm. Similarly using (\ref{eq:approx 3}),
\begin{align}\label{eq:approxCW3}
\|\left(<\mathbf{c},\mathbf{I}-\mathbf{I}_{\epsilon}>\right)\|_{L^2(\Omega)}\le  \|\mathbf{c}\| \|\mathbf{I}-\mathbf{I}_{\epsilon}\|_{L^2(\Omega)}\le  \|\mathbf{c}\|\sqrt{J_L}\epsilon .
\end{align}

We now apply a usual triangular approximation argument (e.g., Lemma 4.2.1 of \citep{giraitis2009large}). Let
\begin{align*}
U_N^{(m)}(t)&=<\mathbf{a},\mathbf{Y}^{(m)}_{N,S_1}(t)>+<\mathbf{b},\mathbf{Y}^{(m)}_{N,S_2}(t)>+<\mathbf{c},\mathbf{Q}_{N}>,\\
U_{N,\epsilon}^{(m)}(t)&=<\mathbf{a},\mathbf{Y}^{(m)}_{N,S_1}(t)>+<\mathbf{b},\mathbf{Y}^{(m)}_{N,S_2}(t)>+<\mathbf{c},\mathbf{Q}_{\epsilon,N}>,\\
U_\epsilon^{(m)}(t)&=<\mathbf{a},\mathbf{W}(t)>+<\mathbf{b},\mathbf{B}^{(m)}(t)>+<\mathbf{c},\mathbf{I}_\epsilon>,\\
U^{(m)}(t)&=<\mathbf{a},\mathbf{W}(t)>+<\mathbf{b},\mathbf{B}^{(m)}(t)>+<\mathbf{c},\mathbf{I}>.
\end{align*}
By (\ref{eq:approxCW2}),  (\ref{eq:approxCW3}) and (\ref{eq:approxCW1}), we have that
\begin{align*}
&U_{N,\epsilon}^{(m)}(t)\ConvFDD X_{\epsilon}^{(m)}(t)\text{ as }N\rightarrow \infty,\\
&U_{\epsilon}^{(m)}(t)\ConvFDD X^{(m)}(t) \text{ as } \epsilon\rightarrow 0,\\
&\underset{ \epsilon\rightarrow 0}{\lim} \underset{N\rightarrow\infty}{\limsup} \|U_N^{(m)}(t)-U_{N,\epsilon}^{(m)}(t)\|_{L^2(\Omega)}=0,~ \forall~ t\ge 0,
\end{align*}
which implies  $U^{(m)}_N(t)\ConvFDD U^{(m)}(t)$,
proving (\ref{eq:target}).
\end{proof}

The next lemma gets rid of the $m$-truncation.
\begin{Lem}\label{Lem:final}
\cref{Lem:2} holds with the m-truncated normalized partial sums $\mathbf{Y}^{(m)}_{N,S_i}(t),~i=1,2$ replaced with the non-truncated ones: $$\mathbf{Y}_{N,S_i}(t)=\left(\frac{1}{\sigma_{j,S_i}\sqrt{N}}\sum_{n=1}^{[Nt]}X_{j,S_i}(n)\right)_{j=1,\ldots,J_i},~i=1,2, $$
where $X_{j,S_i}(n)$ is the non-truncated multilinear polynomial-form process corresponding to the component of $\mathbf{Y}_{N,S_i}$ in \cref{Thm:SRD&LRD}, $\sigma_{j,S_i}:=\sum_n\Cov(X_{j,S_i}(n),X_{j,S_i}(0))$
and the limit  $\mathbf{B}^{(m)}(t)$ is replaced  by  $\mathbf{B}(t)$, that is, as $N\rightarrow\infty$,
\begin{align}\label{eq:LemFinal}
\Big(\mathbf{Y}_{N,S_1}(t), \mathbf{Y}_{N,S_2}(t),\mathbf{Q}_N\Big) \ConvFDD \Big(\mathbf{W}(t), \mathbf{B}(t),\mathbf{I}\Big),
\end{align}
where $\mathbf{W}(t)=\left(W(t),\ldots,W(t)\right)$, $\mathbf{B}(t)=\left(B_1(t),\ldots,B_{J_{S_2}}(t)\right)$ are as given in \cref{Thm:SRD&LRD}.
\end{Lem}
\begin{proof}

We apply again the triangular argument at the end of the proof of \cref{Lem:2} above, but now with $m\rightarrow\infty$, namely, to show $U_N(t)\ConvFDD U(t)$, we show
\begin{align*}
&U_N^{(m)}(t)\ConvFDD U^{(m)}(t)\text{ as }N\rightarrow \infty,\\
&U^{(m)}(t)\ConvFDD U(t) \text{ as } m\rightarrow \infty,\\
&\underset{ m\rightarrow \infty}{\lim} \underset{N\rightarrow\infty}{\limsup} \|U_N^{(m)}(t)-U_{N}(t)\|_{L^2(\Omega)}=0,~ \forall~ t\ge 0,
\end{align*}
 The first step follows from \cref{Lem:2}. The second follows from (\ref{eq:Sigma Conv}) since that relation implies that the Gaussian vector   $(\mathbf{W},\mathbf{B}^{(m)}(t))$   converges to $(\mathbf{W},\mathbf{B}(t))$. For the last step, apply the argument leading to (4.8.7) of \citep{giraitis2009large} and hence for any $t\ge 0$ as $N\rightarrow\infty$,
\begin{equation}\label{eq:m->inf}
\|Y^{(m)}_{N,j,i}(t)-Y_{N,j,i}(t)\|_{L^2(\Omega)}\rightarrow 0,~j=1,\ldots,J_{S_i},~i=1,2.
\end{equation}
\end{proof}

Now we prove \cref{Thm:SRD&LRD}:
\begin{proof}[Proof of \cref{Thm:SRD&LRD}]
In view of \cref{Lem:final}, it is only necessary to verify that the assumption on $\mathbf{Q}_N$ are satisfied, that is, we  now focus on the LRD component: $$\mathbf{Y}_{N,L}(t)=\left(\frac{1}{A_{p,L}(N)}\sum_{n=1}^{[Nt]}X_{p,L}(n)\right)_{p=1,\ldots, J_L}$$ in \cref{Thm:SRD&LRD}.
Choose as kernels $h_{p,N}$ in (\ref{eq:h LRD}) those obtained from  $\mathbf{Y}_{N,L}$, that is,
\[
h^{(t)}_{p,N}(s_1,\ldots,s_{k_{p,L}})=c(p,N) N^{-1+k_{p,L}(1/2-d_{p,L})}\sum_{n=1}^{[Nt]} \prod_{i=1}^{k_{p,L}} a_{n-s_i,p},
\]
where $c(p,N)>0$ is some normalization constant. By Theorem 4.8.2 of \citep{giraitis2009large}, (\ref{eq:kernelConv}) holds and so therefore does \cref{Lem:final}. This concludes the proof of \cref{Thm:SRD&LRD}.
\end{proof}

\subsection{Weak convergence in $D[0,1]^J$}
We first state a lemma which will be used to prove case \ref{case:multilinear}.
\begin{Lem}\label{Lem:HyperContract}
Let $Q_k(h)$ be a polynomial form defined in (\ref{eq:Q_k(h)}). If  \begin{equation}\label{eq:abs sum h}
\sum_{i_1,\ldots,i_k}|h(i_1,\ldots,i_k)|<\infty,
\end{equation}
and $E(|\epsilon_i|^5)<\infty$,  then we have the following hypercontractivity inequality:
\begin{align}\label{eq:hypercontract discrete}
E\left(Q_k(h)^4\right)\le c E\left(Q_k(h)^2\right)^2,
\end{align}
where $c=\left(3+2E(\epsilon_i^4)\right)^{2k}$.
\end{Lem}
\begin{proof}
 Let $h_M$ be the truncated version of $h$, that is, $$h_M(i_1,\ldots,i_k)=h(i_1,\ldots,i_k)\mathbf{1}_{\{i_1\le M,\ldots,i_k\le M\}}(i_1,\ldots,i_k).$$
By the absolute summability of $h$, we have $\E\left(|Q_k(h_M)-Q_k(h)|\right)\le (E|\epsilon_i|)^k\sum_{i_1> M,\ldots,i_k> M}|h(i_1,\ldots,i_k)|\rightarrow 0$ as $M\rightarrow\infty$, and thus
\begin{equation}\label{eq:Q_k(h_M)->Q_k(h)}
Q_k(h)\ConvD Q_k(h).
\end{equation}
By (11.4.1) of \citep{nourdin2012normal}, we have for $M\ge k$,
\begin{align}\label{eq:hypercontrac M}
E\left(Q_k(h_M)^4\right)\le \left(3+2E(\epsilon_i^4)\right)^{2k} E\left(Q_k(h_M)^2\right)^2
\end{align}
In addition,
\begin{equation}\label{eq:Q^5 bound}
\E\left(|Q_k(h_M)|^5\right)\le A\left(\sum_{i_1,\ldots,i_k} |h(i_1,\ldots,i_k)|\right)^5<\infty,
\end{equation}
where $A>0$ is a constant accounting for the product of absolute moments of $\{\epsilon_i\}$. Note that since $h$ vanishes on the diagonals $i_p=i_q$ when $p\neq q$,   there is no moment-order higher than  $5$ involved there.

Finally, (\ref{eq:Q^5 bound}) implies that $\{Q_k(h_M)^4, M\ge 1\}$ and $\{Q_k(h_M)^2, M\ge1\}$ are uniformly integrable, and this combined with (\ref{eq:Q_k(h_M)->Q_k(h)}) and (\ref{eq:hypercontrac M}) yields (\ref{eq:hypercontract discrete}).
\end{proof}

\begin{proof}[Proof of \cref{Thm:WeakConv}]
Convergence in finite-dimensional distributions follows from
\cref{Thm:PureSRD}, \cref{Thm:PureLRD} and \cref{Thm:SRD&LRD}, so we are left to show tightness in $D[0,1]^J$. Since univariate tightness implies the multivariate tightness in the product space (Lemma 3.10 of \citep{taqqu2012multivariate}), we only need to show that each  $\{Y_{j,N}(t), N\ge 1\}$ in (\ref{eq:Y_j,N}) is tight with respect to the uniform metric. If $X_j(n)$ is LRD, the tightness is shown in Theorem 4.8.2 of \citep{giraitis2009large}. We  only need to treat the SRD case.

Suppose that $X(n)$ is a process defined in (\ref{eq:Def Poly Process}) which is SRD.

In case \ref{case:m-dependent} of \cref{Thm:WeakConv}, note that $X_n$ is now a stationary $m$-dependent sequence, so the weak convergence of $S_N(t)$ to Brownian motion, which includes tightness, is classical (\citep{billingsley1956invariance} Theorem 5.2).

Consider next case \ref{case:gaussian}. Because $\epsilon_i$ are i.i.d.\ Gaussian, $X(n)$ belongs to the $k$-th Wiener chaos, or say, can be written as a multiple Wiener-It\^o integral of order $k$ (see, e.g., \citep{nourdin2012normal} Chapter 2.2 and Chapter 2.7). Since the $k$-th Wiener chaos is a linear space,  $Y_N(t):=\frac{1}{\sqrt{N}}\sum_{n=1}^{[Nt]}X(n)$ also belongs to the $k$-th Wiener chaos, and so does $Y_N(t)-Y_N(s)$ for any $0\le s<t$. By the hypercontractivity inequality (Theorem 2.7.2 in \citep{nourdin2012normal}), we have
\begin{equation}\label{eq:hypercontract}
E[|Y_N(t)-Y_N(s)|^4]\le  c E[|Y_N(t)-Y_N(s)|^2]^2,
\end{equation}
where $c$ is some constant which doesn't depend on $s,t$ or $N$. Note that $\sum_n |\gamma(n)|<\infty$  due to SRD assumption, we have
\begin{align}\label{eq:secondmoment}
&E[|Y_N(t)-Y_N(s)|^2]=\frac{1}{N}E[|\sum_{n=1}^{[Nt]-[Ns]}X(n)|^2]
\notag\\=& \frac{[Nt]-[Ns]}{N}\sum_{n=-([Nt]-[Ns])}^{[Nt]-[Ns]}\left(1-\frac{|n|}{[Nt]-[Ns]}\right)\gamma(n)\le \frac{[Nt]-[Ns]}{N}\sum_{n=-\infty}^\infty |\gamma(n)|.
\end{align}
Combining (\ref{eq:hypercontract}) and (\ref{eq:secondmoment}), we  have for some constant $C>0$ that
\[
E[|Y_N(t)-Y_N(s)|^4]\le cE[|Y_N(t)-Y_N(s)|^2]^2\le C |F_N(t)-F_N(s)|^2,
\]
where $F_N(t)=[Nt]/N$. Now by applying Lemma 4.4.1 and Theorem 4.4.1 of \citep{giraitis2009large}, we conclude that tightness holds.

Case \ref{case:linear} is shown by Proposition 4.4.4 of \citep{giraitis2009large} with $H=1/2$.

For case \ref{case:multilinear}, for $s<t$,
$$
\frac{1}{A(N)}\sum_{n=1}^{[Nt]-[Ns]}X(n)=\sum_{1\le i_1<\ldots<i_k<\infty} \left(\frac{1}{A(N)} \sum_{n=1}^{[Nt]-[Ns]}a_{n-i_1}\ldots a_{n-i_k}\right) \epsilon_{i_1}\ldots\epsilon_{i_k},
$$
Thus \cref{Lem:HyperContract} applies with $h(i_1,\ldots,i_k)=\frac{1}{A(N)}\sum_{n=1}^{[Nt]-[Ns]}a_{n-i_1}\ldots a_{n-i_k}$ since (\ref{eq:abs sum h}) holds due to the assumption $\sum_{i\ge 1}|a_i|<\infty$. Tightness then follows by applying the same argument as in case \ref{case:gaussian}.
\end{proof}

\medskip

\noindent\textbf{Acknowledgments.} This work was partially supported by the NSF grant DMS-1007616 at Boston University.
\bibliographystyle{plainnat}
%\bibliography{Bib}

\bigskip

\noindent Shuyang Bai~~~~~~~ \textit{bsy9142@bu.edu}\\
Murad S. Taqqu ~~\textit{murad@bu.edu}\\
Department of Mathematics and Statistics\\
111 Cumminton Street\\
Boston, MA, 02215, US

\end{document}